\documentclass{amsart}

\usepackage{amsmath,xypic,leftidx}
\usepackage{xspace}
\usepackage[psamsfonts]{amssymb}
\usepackage[latin1]{inputenc}
\usepackage{graphicx,color}

\usepackage{hyperref}

\theoremstyle{remark}

\theoremstyle{plain}

\newtheorem*{theorem}{\bf Theorem}

\newtheorem{lemma}{\bf Lemma}
\newtheorem{corollary}{\bf Corollary}

\def\C{{\mathbb C}}

\def\Z{{\mathbb Z}}
\def\H{{\mathbb H}}

\def\P{{\widehat {\mathbb C}}}
\def\T{{\mathcal T}}

\def\x{{\bf x}}
\def\y{{\bf y}}
\def\v{{\bf v}}
\def\w{{\bf w}}
\def\G{{\bf G}}
\def\X{{\frak X}}

\def\and{{\quad\text{and}\quad}}
\def\with{{\quad\text{with}\quad}}

\def\eqdef{{:=}}

\title{Perturbations of flexible Latt\`es maps}
\begin{author}[X.~Buff]{Xavier Buff}
\email{xavier.buff$@$math.univ-toulouse.fr}\thanks{The research of the first author was supported in part by the grant ANR-08-JCJC-0002 and the IUF}
\end{author}
\begin{author}[T.~Gauthier]{Thomas Gauthier}
\email{thomas.gauthier$@$math.univ-toulouse.fr}
\address{ %
  Universit\'e Paul Sabatier\\
  Institut de Math\'ematiques de Toulouse \\
  118, route de Narbonne \\
  31062 Toulouse Cedex \\
  France }
\end{author}

\begin{document}

\begin{abstract}
We prove that any Latt\`es map can be approximated by  
strictly postcritically finite rational maps which are not Latt\`es maps. 
\end{abstract}

\maketitle

\section*{Introduction}

A rational map of degree $D\geq 2$ is {\em strictly postcritically finite} if the orbit of each critical point intersects a repelling cycle. Among those, flexible Latts maps (the definition is given below) play a special role. The following result answers a question raised in \cite{be}. 

\begin{theorem}
Every flexible Latt\`es map can be approximated by strictly postcritically finite rational maps which are not Latt\`es maps. 
\end{theorem}

Given $D\geq 2$, denote by ${\rm Rat}_D$ the space of rational maps of degree $D$. A rational map $f\in {\rm Rat}_D$ has a {\em Julia set} $J_f$ which may be defined as the closure of the set of repelling cycles of $f$. The Julia set is the support of a measure $\mu_f$ which may be defined as the unique invariant measure of (maximal) entropy $\log D$.

The {\em bifurcation locus} in ${\rm Rat}_D$ is the closure of the set of discontinuity of the map $f\mapsto J_f$. Laura DeMarco \cite{dm} proved that the bifurcation locus is the support of a positive closed $(1,1)$-current $T_{\rm bif}\eqdef dd^c {\frak L}$, where ${\frak L}$  is the plurisubharmonic function which sends a rational map $f$ to its Lyapunov exponent with respect to $\mu_f$. 

M\"obius transformations act by conjugacy on ${\rm Rat}_D$ and the quotient space is an orbifold known as the {\em moduli space} $\mathcal{M}_D$ of rational maps of degree $D$. Giovanni Bassanelli and Fran\c{c}ois Berteloot \cite{bb} introduced a measure $\mu_{\rm bif}$ on this moduli space, which may be obtained by pushing forward $T_{\rm bif}^{\wedge (2D-2)}$. 

In \cite{be}, the first author and Adam Epstein, using a transversality result in ${\rm Rat}_D$, proved that the conjugacy class of a strictly postcritically finite map which is not a flexible Latt\`es map is in the support of $\mu_{\rm bif}$. Since the support of $\mu_{\rm bif}$ is closed and since the conjugacy class of a strictly postcritically finite rational maps which is not a Latt\`es maps is in this support, our result has the following consequence. 

\begin{corollary}
The classes of flexible Latt\`es maps in $\mathcal{M}_D$ lie in the support of the bifurcation measure $\mu_{\rm bif}$.
\end{corollary}

In \cite{g}, the second author proved that the support of the bifurcation measure has maximal Hausdorff dimension, i.e. has dimension $2(2D-2)$, and that it is homogeneous (the support of $\mu_{\rm bif}$ has maximal dimension in any neighborhood of its points). Corollary $1$ thus yields the following result.

\begin{corollary}
Let $f\in{\rm Rat}_D$ be a flexible Latt\`es map and let $V\subset\mathcal{M}_D$ be an open neighborhood of the conjugacy class of $f$. Then, $\dim_H({\rm supp}(\mu_{\rm bif})\cap V)=2(2D-2)$.
\end{corollary}

Bassanelli and Berteloot \cite{bb2} proved that every point in the support of $\mu_{\rm bif}$ can be approximated by rational maps having $2D-2$ distinct neutral cycles. Their argument can be adapted to prove that the support of $\mu_{\rm bif}$ can be approximated by hyperbolic maps having $2D-2$ distinct attracting cycles (see \cite{courseBC} Section 5.2). By Corollary $1$, we have the following result.

\begin{corollary}
Any flexible Latt\`es map $f\in{\rm Rat}_D$ can be approximated by hyperbolic rational maps having $2D-2$ distinct attracting cycles.
\end{corollary}

The approach for solving this problem was suggested by John Milnor. We wish to express our gratitude. We also wish to thank the Banff International Research Station for hosting the workshop ``Frontiers in Complex Dynamics (Celebrating John Milnor's 80th birthday)'' during which we developed our proof. 

\section{Flexible Latt\`es maps}

Following Milnor \cite{milat}, we define a {\em flexible Latt\`es map} of degree $D\geq 2$ to be a rational map $f:\P\to \P$ for which there is a commutative diagram :
\[\diagram
\C/\Lambda \rto^{L} \dto_\Theta & \C/\Lambda  \dto^\Theta\\
\P\rto_f &\P
\enddiagram\]
where
\begin{itemize}
\item $\Lambda\subset \C$ is a lattice of rank $2$;
\item $\T\eqdef \C/\Lambda$ is the quotient torus;
\item $L:\T\ni \tau\mapsto a\tau+b\in  \T$   with $a\in \Z$, $a^2=D$, and $2b\in \Lambda/\bigl(2\Lambda+(a-1)\Lambda\bigr)$;
 \item $\Theta:\T\to \P$ is a $2$-to-$1$ holomorphic map ramifying at points in $\Lambda/2$. 
\end{itemize}

In the rest of the article, the lattice $\Lambda$ will be of the form $\Lambda=\Z\oplus \gamma\Z$, where $\gamma$ is a complex number which is allowed to vary in the upper half-plane $\H$ of complex numbers with positive imaginary part. We shall use the following normalization for $\Theta_\gamma:\T\to \P$ : 
\[\Theta_\gamma(0)=0,\quad \Theta_\gamma\left(\frac{\gamma+1}{2}\right)=\infty,\quad \Theta_\gamma\left(\frac{1}2\right)=v(\gamma)\quad \text{and}\quad \Theta_\gamma\left(\frac{\gamma}2\right)=w(\gamma)\] where  $v:\H\to  \C-\{0\}$ and $w:\H\to \C-\{0\}$ are holomorphic functions (this may be achieved for example by requiring that $v\equiv 1$). Then, for all $\gamma\in \H$, $v(\gamma)\neq w(\gamma)$. 

~

The derivative of the  torus endomorphism $L_\gamma:\T_\gamma \to \T_\gamma$ will be a fixed integer $a$ which does not depend on $\gamma$. When $a=\pm 2$, the critical value set of the Latt\`es map $f_\gamma:\P\to \P$ is $\bigl\{\infty,v(\gamma),w(\gamma)\bigr\}$ and if $|a|\geq 3$, the critical value set is $\bigl\{0,\infty,v(\gamma),w(\gamma)\bigr\}$. In all cases, the postcritical set of $f_\gamma$ is $\bigl\{0,\infty,v(\gamma),w(\gamma)\bigr\}$. More precisely, we will have to study three cases : 
\begin{itemize}
\item Case 1 : $a$ is even and $L_\gamma(\tau) = a\tau$. In that case, all the critical values are mapped to $0$ which is a fixed point of $f$. 

\item Case 2 : $a$ is odd and $L_\gamma(\tau) = a\tau$. In that case, all the critical values are fixed with multiplier $a^2$. 

\item Case 3 : $a$ is odd and $L_\gamma(\tau) = a\tau + (\gamma+1)/2$. In that case, the Latt\`es map permutes the critical value at $0$ with that at infinity. It also permutes the critical value at $v(\gamma)$ with the critical value at $w(\gamma)$. The multiplier of each cycle is $a^4$. 
\end{itemize}
From now on, we assume that we are in one of those three cases, and we consider the analytic family of Latt\`es maps 
\[\H\ni \gamma\mapsto f_\gamma\in {\rm Rat}_D,\]
where $ {\rm Rat}_D$ is the space of rational maps of degree $D$. 
We shall use the notation $f$, $v$, $w$, \ldots  in place of $f_\gamma$, $v(\gamma)$, $w(\gamma)$, \ldots when $\gamma$ is assumed to be fixed and there is no confusion.

\section{Estimates for $\Theta$}

\begin{lemma}\label{lemma:c1c2} 
As $\tau\to 0$, we have the following expansion
\[\Theta(1/2+\tau) = v+\lambda \tau^2+o(\tau^2)\quad \text{and}\quad \Theta(\gamma/2+\tau) = w+\mu \tau^2+o(\tau^2)\]
with
\[\frac{\lambda}{v} =-\frac{\mu}{w}\neq 0.\]
\end{lemma}

\begin{proof}
Since $\Theta$ has simple critical points at $1/2$ and $\gamma/2$, we have an expansion as in the statement with $\lambda\neq 0$ and $\mu\neq 0$. Our work consists in proving the relation between $\lambda/v$ and $\mu/w$. 
Let $q$ be the meromorphic quadratic differential on $\P$ defined by :
\[q \eqdef  \frac{dz^2}{z(z-v)(z-w)}.\]
Since $q$ has simple poles and since $\Theta$ is totally ramified above the polar set of $q$, the quadratic differential $\Theta^*q$ is holomorphic on $\T$, whence 
\[\Theta^* q=\kappa\cdot d\tau^2\quad \text{with}\quad \kappa\in \C-\{0\}.\] 
Let $\tau$ tend to $0$ and set $z=\Theta(1/2+\tau)$. Then, 
$z- v\sim \lambda \tau^2 $ and $dz^2 \sim 4\lambda ^2 \tau^2 d\tau^2$. 
It follows that 
\[ \frac{dz^2}{z(z-v)(z-w)} \sim \frac{4\lambda ^2\tau^2d\tau^2}{v\cdot (v-w)\cdot \lambda \tau^2} = 
 \frac{4\lambda d\tau^2}{v(v-w)}.\]
This show that 
\[\kappa = \frac{4\lambda }{v(v-w)}.\]
Using the expansion of $\Theta$ at $\gamma/2$ yields similarly
\[\kappa = \frac{4\mu }{w(w-v)}.\]
The result now follows easily. 
\end{proof}

\section{Hyperbolic sets}

From now on, we let $\alpha$, $\alpha'$, $\beta$ and $\beta'$ be rational numbers. We assume that the denominators of $\alpha$ and $\beta$ are coprime with $a$ and $2$.  We let $\sigma$ and $\tau$ be the points in the torus $\T$ defined by
\[\sigma \eqdef \alpha+\alpha'\gamma\quad \text{and}\quad \tau \eqdef \beta+\beta'\gamma.\]
Remark that, as $\sigma$ and $\tau$ have rational coordinates, they are (pre)periodic under multiplication by $a$ and that $\sigma$, $\tau$, $a\sigma$ and $a\tau$ are (pre)periodic under multiplication by $a^2$. For each integer $k\geq  1$, let $\tau_k$ and $\sigma_k$ be the points in $\T$ defined by  
\[\sigma_k \eqdef  \frac{1}{2}+\frac{\sigma}{a^{k}}\quad \text{and}\quad \tau_k\eqdef  \frac{\gamma}{2}+\frac{\tau}{a^{k}}\]
and let $x_k$ and $y_k$ be the points in $\P$ defined by 
\[x_k\eqdef \Theta(\sigma_k)\quad \text {and}\quad y_k\eqdef \Theta(\tau_k).\]
We shall denote by $X$, $Y$, and $Z$ the following subsets of $\P$ :
\[X\eqdef \bigcup_{k\geq 1}\bigcup_{n\geq 0} f^{\circ n}(x_k),\quad 
Y\eqdef \bigcup_{k\geq 1}\bigcup_{n\geq 0} f^{\circ n}(y_k)\quad \text{and}\quad
Z \eqdef X\cup Y\cup \{0,\infty,v,w\}.\]
Note that $Z$ is an invariant set. 

\begin{lemma}\label{lemma:holomot}
There is a neighborhood $\G$ of $f$ in ${\rm Rat}_D$ and a dynamical holomorphic motion $\phi:\G\times Z\to \P$ such that for all $g\in \G$ and all $z\in Z$, we have 
\[g\bigl( \phi(g,z)\bigr)= \phi\bigl(g, g(z)\bigr).\]
\end{lemma}

\begin{proof}
\noindent{\bf Case 1 :} $a$ is even. In that case, $f^{\circ k}$ maps $x_k$ to $\Theta(\sigma)$ and $y_k$ to $\Theta(\tau)$. Since $\sigma$ and $\tau$ are (pre)periodic under the action of $L$,  the points $x_k$ and $y_k$ are (pre)periodic.

~

\noindent{\bf Case 2 :} $a$ is odd and $b=0$. In that case, $f^{\circ k}$ maps $x_k$ to $\Theta(\frac{1}{2}+\sigma)$ and $y_k$ to $\Theta(\frac{\gamma}{2}+\tau)$. As $L(\frac{1}{2}+\sigma)=\frac{1}{2}+a\sigma$ and $L(\frac{\gamma}{2}+\tau)=\frac{\gamma}{2}+a\tau$ and as $\sigma$ and $\tau$ are (pre)periodic under multiplication by $a$, the points $\frac{1}{2}+\sigma$ and $\frac{\gamma}{2}+\tau$ are (pre)periodic under the action of $L$, which means that the points $x_k$ and $y_k$ are (pre)periodic.

~

\noindent{\bf Case 3 :}  $a$ is odd and $b=(\gamma+1)/2$. In that case, we use that
\[L^{\circ 2}\Big(\frac{1}{2}+z\Big)=\frac{1}{2}+a^2z\quad \text{and}\quad L^{\circ 2}\Big(\frac{\gamma}{2}+z\Big)=\frac{\gamma}{2}+a^2z\]
for any $z\in\mathcal{T}$. If $k$ is even, $f^{\circ k}$ maps $x_k$ to $\Theta(\frac{1}{2}+\sigma)$ and $y_k$ to $\Theta(\frac{\gamma}{2}+\tau)$. Since $\sigma$ and $\tau$ are (pre)periodic under multiplication by $a^2$, $\frac{1}{2}+\sigma$ and $\frac{\gamma}{2}+\tau$ are (pre)periodic under iteration of $L^{\circ 2}$ and the points $x_k$ and $y_k$ are (pre)periodic.
\par Now, if $k$ is odd, $f^{\circ k}$ maps $x_k$ to $\Theta(\frac{\gamma}{2}+\sigma)$ and $y_k$ to $\Theta(\frac{1}{2}+\tau)$. Since $L(\frac{\gamma}{2}+\sigma)=\frac{1}{2}+a\sigma$ and $L(\frac{1}{2}+\tau)=\frac{\gamma}{2}+a\tau$ and since $a\sigma$ and $a\tau$ are (pre)periodic under multiplication by $a^2$, $\frac{1}{2}+a\sigma$ and $\frac{\gamma}{2}+a\tau$ are (pre)periodic under the action of $L^{\circ 2}$, which implies that the points $x_k$ and $y_k$ are (pre)periodic.

~

In all cases, the set $Z$ is closed and each point in $Z$ is (pre)periodic to a repelling cycle of $f$. Since  the denominators of $\alpha$ and $\beta$ are coprime with $a$ and $2$, the cycles capturing points in $X$ and $Y$ are disjoint from the postcritical set of $f$. In particular, $Z$ does not contain critical points. As a consequence,  $Z$ is a hyperbolic invariant compact set (see \cite{shishitan} Theorem 1.2 page 266) and the result follows (see \cite{shishi} $(1.2)$ section 2 for a sketch of the proof, or \cite{gthese} section 2.1 for a proof). 
\end{proof}

For all $g\in \G$, the map $\phi_g:z\mapsto \phi(g,z)$ is a homeomorphism from  $Z$ to $Z_g\eqdef \phi_g(Z)$ which conjugates $f:Z\to Z$ to $g:Z_g\to Z_g$. Every point $z\in Z_g$ is (pre)periodic under iteration of $g$ to a repelling cycle of $g$. For $k\geq 1$, we shall use the notation $\x_k(g)\eqdef \phi_g(x_k)$ and $\y_k(g)\eqdef \phi_g(y_k)$. 
As $k\to +\infty$, the sequence $x_k$ tends to $x_\infty\eqdef v$. We set $\x_\infty(g)\eqdef \phi_g(v)$. Similarly, 
as $k\to +\infty$, the sequence $y_k$ tends to $y_\infty\eqdef w$. We set $\y_\infty(g)\eqdef \phi_g(w)$.

\section{A family of perturbations}

Let us now restrict to the family of rational maps $t\mapsto f_t$ which is defined for $t\in \C$ by 
\[f_t\eqdef (1+t)\cdot f.\]
If $t\neq -1$, the rational map $f_t$ has degree $D$ and $f_0= f$. 
The critical values of the rational map $f_t$ are contained in the set $\{0,\infty,(1+t)\cdot v,(1+t)\cdot w\}$ and for $t$ sufficiently close to $0$, the point $0$ is either a repelling fixed point or a repelling periodic point of period $2$, and $\infty$ is either mapped to a repelling fixed point at $0$, or a repelling fixed point, or a repelling periodic point of period $2$. 
We will now study the relative motion of the critical values  $\v(f_t)\eqdef (1+t)\cdot v$ and $\w(f_t)\eqdef (1+t)\cdot w$ with respect to the points $\x_\infty(f_t)$ and $\y_\infty(f_t)$. We will use the notations
\[\dot \v\eqdef \frac{d \v(f_t)}{dt}\Bigl|_{t=0}\in T_{v}\P,\quad \dot \w\eqdef \frac{d \w(f_t)}{dt}\Bigl|_{t=0}\in T_{w}\P,\]
\[ \dot \x_\infty\eqdef \frac{d \x_\infty(f_t)}{dt}\Bigl|_{t=0}\in T_{v}\P,\and \dot \y_\infty\eqdef \frac{d \y_\infty(f_t)}{dt}\Bigl|_{t=0}\in T_{w}\P.\]
We shall also use the notation 
\[\dot f \eqdef \frac{d f_t}{dt}\Bigl|_{t=0}\]
which is a section of the pullback bundle $f^\star T\P$ : for all $z\in \P$, $\dot f(z)\in T_{f(z)}\P$. 

\begin{lemma}\label{lemma:v1v2} 
We have
\[\frac{\dot \x_\infty -\dot \v}{v} = \frac{\dot \y_\infty-\dot \w}{w}\neq 0.\]
\end{lemma}

\begin{proof}
By construction, 
\[\dot \v = v\frac{d}{dz}\and \dot \w = w\frac{d}{dz}.\] 
In addition, for all $z\in \P$, 
\[\dot f(z)=f(z)\frac{d}{dz} .\]

\noindent{\bf Case 1 :} $a$ is even. In that case, $f$ maps $x_\infty=v$ to $0$ which is a fixed point. As $t$ varies, $f_t$ still maps $v$ to $0$ and still fixes $0$. Thus, $\x_\infty\equiv v$ and $\dot \x_\infty = 0$,
whence 
\[\dot \x_\infty -\dot \v = -\dot \v.\]
Similarly 
\[\dot \y_\infty - \dot \w = -\dot \w.\]

\noindent{\bf Case 2 :} $a$ is odd and $b=0$. In that case,  $f$ fixes $x_\infty=v$. So, $\x_\infty(f_t)$ is the fixed point of $f_t$ which is close to $v$. Derivating $f_t\bigl(\x_\infty(f_t)\bigr)=\x_\infty(f_t)$ with respect to $t$, and evaluating at $t=0$, we get :
\[\dot f(v) + D_{v}f(\dot \x_\infty) = \dot \x_\infty.\]
Since $\dot f(v)=v\frac{d}{dz}=\dot \v$ and $D_{v}f$ is multiplication by $a^2$, we get
\[\dot \x_\infty = \frac{\dot \v}{1-a^2}\]
whence 
\[\dot \x_\infty -\dot \v = \left(\frac{1}{1-a^2}-1\right) \dot \v =\frac{a^2}{1-a^2}\dot \v.\]
Similarly
\[\dot \y_\infty - \dot \w = \frac{a^2}{1-a^2} \dot \w.\]

\noindent{\bf Case 3 :} $a$ is odd and $b=(\gamma+1)/2$.  In that case $f$ permutes $x_\infty=v$ and $y_\infty=w$. Derivating $f_t\bigl(\x_\infty(f_t)\bigr)=\y_\infty(f_t)$ and $f_t\bigl(\y_\infty(f_t)\bigr)=\x_\infty(f_t)$ with respect to $t$,  and evaluating at $t=0$, we get :
\begin{equation}\label{eq:case3}
\dot f(v) + D_{v}f(\dot \x_\infty) = \dot \y_\infty\and \dot f(w) + D_{w}f(\dot \y_\infty) = \dot \x_\infty.
\end{equation}
On the one hand $\dot f(v) = w\frac{d}{dz}=\dot \w$ and $\dot f(w)=v\frac{d}{dz}=\dot \v$. 
On the other hand, 
\[L\left(\frac{1}{2}+\tau\right) = \frac{\gamma}{2}+a\tau\and L\left(\frac{\gamma}{2}+\tau\right) = \frac{1}{2}+a\tau.\]
Thus, 
\[f\bigl(v+\lambda \tau^2+o(\tau^2)\bigr)= w+\mu a^2\tau^2+o(\tau^2)\and
f\bigl(w+\mu \tau^2+o(\tau^2)\bigr)= v+\lambda a^2\tau^2+o(\tau^2)\]
which implies
\[\lambda f'(v) = \mu a^2\and \mu f'(w) = \lambda a^2.\]
By Lemma \ref{lemma:c1c2}, it follows that 
\[D_{v}f(\dot \x_\infty) = \frac{\mu a^2}{\lambda }\dot \x_\infty = - \frac{a^2 w}{v}\dot \x_\infty.\]
Similarly, 
\[D_{w}f(\dot \y_\infty) =- \frac{a^2 v}{w}\dot \x_\infty.\]
So, equation (\ref{eq:case3}) yields
\[\dot \w -  \frac{a^2 w}{v} \dot \x_\infty= \dot \y_\infty\quad \text{and}\quad 
\dot \v - \frac{a^2 v}{w} \dot \y_\infty= \dot \x_\infty\]
which boils down to
\[\dot \x_\infty = \frac{1-a^2}{1-a^4}\dot \v\quad \text{and}\quad \dot \y_\infty =\frac{1-a^2}{1-a^4}\dot \w.\]
As a consequence
\[\dot \x_\infty -\dot \v=\left( \frac{1-a^2}{1-a^4}-1\right) \dot \v= -\frac{a^2}{1+a^2}\dot \v.\]
Similarly
\[\dot \y_\infty -\dot \w =-\frac{a^2}{1+a^2}\dot \w.\]

In all cases, there is a constant $c \neq 0$ ($c=-1$ in the first case, $c=a^2/(1-a^2)$ in the second case and $c= -a^2/(1+a^2)$ in the third case) such that 
\[\dot \x_\infty -\dot \v = c \dot \v = c v\frac{d}{dz}\quad \text{and}\quad
 \dot \y_\infty -\dot \w = c \dot \w = c w\frac{d}{dz}.\]
 This shows that 
\[\frac{\dot \x_\infty -\dot \v}{v} = \frac{\dot \y_\infty-\dot \w}{w}=c \frac{d}{dz}\neq 0.\qedhere\]
\end{proof}

\section{Proof of the Theorem}

Let $\gamma_0\eqdef x_0+iy_0 \in \H$ be a point with rational real part $x_0$ and rational imaginary part $y_0$ which are coprime with $a$ and $2$. 
Set $\alpha\eqdef -x_0$, $\alpha' \eqdef 1$, $\beta\eqdef y_0$ and $\beta'\eqdef 0$. We will now allow $\gamma$ to vary in $\H$. So, $\sigma$, $\tau$, $v$, $w$, $\lambda $, $\mu $, $\sigma_k$, $\tau_k$, $x_k$, $y_k$, $x_\infty$, $y_\infty$ are functions of $\gamma$. 
For example, 
\[\sigma:\H\ni \gamma\mapsto \alpha+\alpha'\gamma\in \H\and \tau:\H\ni \gamma\mapsto  \beta+\beta'\gamma\in \H.\]
Note that the M\"obius transformation $\sigma/\tau$ sends $\gamma_0$ to $i$. 

We still consider the family of flexible Latt\`es maps $\H\ni \gamma\mapsto f_\gamma\in {\rm Rat}_D$ and we introduce the family of perturbed maps
\[\H\times (\C-\{-1\}) \ni (\gamma,t) \longmapsto f_{\gamma,t}\eqdef (1+t)\cdot f_\gamma \in {\rm Rat}_D.\]
As in the previous section, we will consider the critical value functions 
\[\v:f_{\gamma,t}\mapsto (1+t)\cdot v(\gamma)\and
\w:f_{\gamma,t}\mapsto (1+t)\cdot w(\gamma).\]
The notation $\dot \v$, $\dot \w$, $\dot \x_\infty$, $\dot \y_\infty$ refers to the derivative with respect to $t$ evaluated at $t=0$. In particular, those are functions of $\gamma\in \H$.

\begin{lemma}
There are a neighborhood $\Gamma$ of $\gamma_0$ in $\H$, an integer $k_0$ and for $k\geq k_0$, holomorphic functions $s_k:\Gamma\to \C$ and $t_k:\Gamma\to \C$ such that : 
\begin{itemize}
\item the sequences $(s_k)$ and $(t_k)$ converge uniformly on $\Gamma$ to $0$,
\item $\x_k(f_{\gamma,s_k(\gamma)})=\v(f_{\gamma,s_k(\gamma)}) $, $\y_k(f_{\gamma,t_k(\gamma)})=\w(f_{\gamma,t_k(\gamma)})$ and 
\item the sequence $(s_k/t_k)$ converges uniformly on $\Gamma$ to $-\sigma^2/\tau^2$.
\end{itemize}
\end{lemma}

\begin{proof}
First, note that as $k\to +\infty$, the holomorphic functions $x_k:\H\to \P$ and $v:\H\to \P$ satisfy
\[x_k-v \sim \lambda \cdot \left(\sigma_k-\frac12\right)^2 = \frac{\lambda \sigma^2}{a^{2k}}.\]
Second, let us consider the sequence of functions 
\[\X_k\eqdef (\gamma,u)\mapsto a^{2k} \cdot  \bigl(\x_k(f_{\gamma,u/a^{2k}}) - \v(f_{\gamma,u/a^{2k}}) \bigr).\]
Note that as $k\to +\infty$,  
\begin{align*}
\X_k(\gamma,u) & = a^{2k}\cdot \left(x_k (\gamma)+  \frac{u}{a^{2k}} \cdot \dot \x_\infty(\gamma) - v(\gamma) -  \frac{u}{a^{2k}} \cdot \dot \v(\gamma) + o(1/a^{2k})\right)\\
 &= \lambda (\gamma)\sigma^2(\gamma) + u\cdot \bigl(\dot \x_\infty(\gamma) -\dot \v(\gamma)\bigr)+ o(1).
 \end{align*}
 So, as $k\to +\infty$, the sequence $(\X_k)$  converges uniformly on every compact subset of $\H\times \C$ to the function
\[\X_\infty:(\gamma,u)\mapsto \lambda \sigma^2 + u\cdot ( \dot \x_\infty-\dot \v).\]
According to Lemma \ref{lemma:v1v2}, $\dot \x_\infty-\dot \v$ does not vanish on $\H$. Thus, the 
function $\X_\infty$ vanishes along the graph of a holomorphic function $u_\infty:\H\to \C$ which satisfies
\[u_\infty \cdot ( \dot \x_\infty-\dot \v) = -\lambda \sigma^2.\]
Since $(\X_k)$ converges locally uniformly to $\X_\infty$, given any neighborhood $\Gamma$ of $\gamma_0$ compactly contained in $\H$, there are, for $k$ large enough, holomorphic functions $u_k:\Gamma\to \C$ which satisfy 
\[(\forall \gamma\in \Gamma)\quad \X_k\bigl(\gamma,u_k(\gamma)\bigr)=0,\]
the sequence $u_k$ converging uniformly on $\Gamma$ to $u_\infty$. 
Now, set 
\[s_k\eqdef \frac{u_k}{a^{2k}}:\Gamma\to \C.\]
Then the sequence $(s_k)$ obviously converges uniformly to $0$ on $\Gamma$ and 
\[\x_k(f_{\gamma,s_k(\gamma)})= \v(f_{\gamma,s_k(\gamma)})\and 
a^{2k} \cdot s_k\cdot  ( \dot \x_\infty-\dot \v) \underset{k\to +\infty}\longrightarrow  -\lambda \sigma^2.\]

Similarly,  there are, for $k$ large enough, holomorphic functions $t_k:\Gamma\to \C$ which converge uniformly on $\Gamma$ to $0$ and satisfy 
\[\y_k(f_{\gamma,t_k(\gamma)})= \w(f_{\gamma,t_k(\gamma)})\and 
a^{2k} \cdot t_k\cdot  ( \dot \y_\infty-\dot \w) \underset{k\to +\infty}\longrightarrow  -\mu \tau^2.\]

Those uniform convergences on $\Gamma$ may be reformulated as
\[a^{2k}\cdot \frac{s_k}{\sigma^2}\underset{k\to +\infty}\longrightarrow-\frac{\dot \x_\infty-\dot \v}{\lambda} 
\and 
 a^{2k}\cdot \frac{t_k}{\tau^2}\underset{k\to +\infty}\longrightarrow - \frac{ \dot \y_\infty-\dot \w}{\mu}.\]
According to Lemma \ref{lemma:c1c2} and Lemma \ref{lemma:v1v2}, 
\[ \frac{\dot \x_\infty-\dot \v}{\lambda} = - \frac{ \dot \y_\infty-\dot \w}{\mu}\]
are non-vanishing functions. 
This shows that the sequence $(s_k/t_k)$ converges uniformly on $\Gamma$ to $-\sigma^2/\tau^2$.
\end{proof}

The previous construction yields, for $k$ large enough, functions $s_k$ and $t_k$ which are defined and holomorphic near $\gamma_0$. 
As $k\to +\infty$, the sequence of functions $(1-s_k/t_k)$ converges to $1+\sigma^2/\tau^2$ which has a simple zero at $\gamma_0$. Therefore, for $k$ large enough, the function $s_k/t_k$ takes the value $1$ at a point $\gamma_k$ close to $\gamma_0$ : 
\[s_k(\gamma_k) = t_k(\gamma_k)\with \gamma_k\underset{k\to +\infty}\longrightarrow \gamma_0.\] 
Let $g_k\in {\rm Rat}_D$ be the rational map 
\[g_k\eqdef (1+r_k)\cdot f_{\gamma_k}\with r_k\eqdef s_k(\gamma_k) = t_k(\gamma_k).\]
Note that $r_k$ tends to $0$ and that $g_k$ converges to $f_{\gamma_0}$ as $k$ tends to $+\infty$. The critical values of $g_k$ are contained in the set $\bigl\{0,\infty,(1+r_k)v,(1+r_k)w\bigr\}$ (if $a=2$, then $0$ is not a critical value). 
By construction, $0$ and $\infty$ are (pre)periodic to repelling cycles. In addition,  $(1+r_k)v$ coincides with $x_k(r_k)$ which is (pre)periodic to a repelling cycle and $(1+r_k)w$ coincides with $y_k(r_k)$ which is also (pre)periodic to a repelling cycle. Thus, $g_k$ is strictly postcritically finite but is not a Latt\`es map. 

We have just proved that $f_{\gamma_0}$ can be approximated by rational maps $g_k$ which are strictly postcritically finite but which are not Latt\`es maps. Since the set of allowable $\gamma_0$ is dense in the upper-half plane, this shows that any flexible Latt\`es map can be approximated by rational maps which are strictly postcritically finite but which are not Latt\`es maps.


\begin{thebibliography}{DeM2}
\bibitem[BB1] {bb} {\sc G. Bassanelli} $\&$ {\sc F. Berteloot},
{Bifurcation currents in holomorphic dynamics on ${\mathbb P}^k$},
J. Reine Angew. Math. 608 (2007), 201--235.

\bibitem[BB2] {bb2} {\sc G. Bassanelli} $\&$ {\sc F. Berteloot},
{Lyapunov exponents, bifurcation currents and laminations in bifurcation loci},
Math. Ann. 345 (2009), 1?23.

\bibitem[B] {courseBC} {\sc F. Berteloot},
{Bifurcation currents in one-dimensional holomorphic dynamics},
Fond. CIME course, Pluripotential theory summer school, 2011, available at http://php.math.unifi.it/users/cime/Courses/2011/course.php?codice=20113.

\bibitem[BE] {be} {\sc X. Buff} $\&$ {\sc A.L. Epstein}, {\em Bifurcation measure and postcritically finite rational maps}, In: Complex dynamics, 491-512, A K Peters, Wellesley, MA (2009). 


\bibitem[DeM] {dm} {\sc L. DeMarco}, {\em Dynamics of rational maps: Lyapunov exponents, bifurcations, and
capacity}, Math. Ann. 326, No.1 (2003), 43--73.

\bibitem[G1] {g} {\sc T. Gauthier}, {\em Strong-bifurcation loci of full Hausdorff dimension}, preprint arXiv http://arxiv.org/abs/1103.2656v1, 2011.

\bibitem[G2] {gthese} {\sc T. Gauthier}, {\em Dimension de Hausdorff de lieux de bifurcations maximales en dynamique des fractions rationnelles}, PhD Thesis, available at http://www.math.univ-toulouse.fr/$\sim$gauthier.

\bibitem[M] {milat} {\sc  J. Milnor}, {\em On Latt\`es maps},
In: Hjorth, P., Petersen, C. L., Dynamics on the Riemann Sphere. A
Bodil Branner Festschrift, European Math. Soc., (2006)


\bibitem[S] {shishi} {\sc  M. Shishikura}, {\em The {H}ausdorff dimension of the boundary of the {M}andelbrot set and {J}ulia sets}, Ann. of Math. (2), 147 (1998), 225--267.

\bibitem[ST] {shishitan} {\sc  M. Shishikura} $\&$ {\sc Lei Tan}, {\em An alternative proof of {M}a\~n\'e's theorem on non-expanding{J}ulia sets}, In: The {M}andelbrot set, theme and variations, volume 274 of London Math. Soc. Lecture Note Ser., 265--279, Cambridge University Press, Cambridge, 2000.


\end{thebibliography}
\end{document}